\documentclass[11pt, reqno]{amsart}

\usepackage[utf8]{inputenc}

\usepackage[margin=0.75in]{geometry}

\usepackage{graphicx}
\usepackage{comment}
\usepackage{fourier}
\usepackage[T1]{fontenc}

\usepackage{amsthm}
\usepackage{amsmath}
\usepackage{amssymb}
\usepackage{mathtools}
\usepackage{apptools}
\usepackage{setspace}
\usepackage{url}
\usepackage{tabularx}

\usepackage[
backend=biber,
style=numeric,
sorting=nyt,
maxnames=10,
giveninits=true
]{biblatex}
\newcommand {\R}{\mathbb{R}}

\newcommand {\Z}{\mathbb{Z}}

\newcommand {\C}{\mathbb{C}}

\newcommand{\re}{\operatorname{Re}}
\newcommand{\im}{\operatorname{Im}}

\newcommand{\Mod}[1]{\ (\mathrm{mod}\ #1)}
\newtheorem{thm}{Theorem}[section]

\newtheorem{assu}[thm]{Assumption}
\newtheorem{lemma}[thm]{Lemma}
\newtheorem{prop}[thm]{Proposition}

\newtheoremstyle{named}{}{}{\itshape}{}{\bfseries}{.}{.5em}{\thmnote{#3}}
\theoremstyle{named}

\AtAppendix{\counterwithin{lemma}{section}}
\numberwithin{equation}{section}

\usepackage{multicol}
\usepackage{xcolor}

\usepackage{hyperref}
\hypersetup{
    colorlinks,
    linkcolor={red!50!black},
    citecolor={blue!50!black},
    urlcolor={blue!80!black},
    backref=page,
}

\usepackage{soul}

\title{Trace formula and functional equation}
\author{Chung-Hang Kwan}
\address{University College London}
\email{ucahckw@ucl.ac.uk}

\author{Wing Hong Leung}
\address{Rutgers University}
\email{joseph.leung@rutgers.edu}

\date{}

\begin{document}

\begin{abstract}
We present a ``beyond-endoscopic'' treatment of the functional equation for the standard $L$-function of a holomorphic cusp form with level and nebentypus. We use Petersson's formula and methods from Venkatesh's thesis and ``spectral reciprocity''.

\end{abstract}

\subjclass[2020]{11F66, 11F72, 11L05, 11F12}
\keywords{Analytic number theory, automorphic form, $L$-function, functional equation, Voronoi formula, Kloosterman sum, Petersson trace formula}

\maketitle

\section{Introduction}

Since Selberg's 1989 Amalfi lectures, there has been a growing interest in exploring how principles from Analytic Number Theory can be applied to study $L$-functions which are axiomatized by their Dirichlet series, Euler products, analytic continuation and functional equations. Notably, Duke and Iwaniec established subconvex bounds for the Dirichlet coefficients of general classes of $L$-functions in a series of papers beginning with \cite{DI1}, where the \emph{functional equations} played a fundamental role and found important automorphic applications.  Their methods were substantially generalized by Luo--Rudnick--Sarnak \cite{LRS-selb} and Blomer--Brumley \cite{BB11}. For other interesting applications of analytic techniques to automorphic forms, see, e.g.,   \cite{SoMul, CoFar}.

\subsection{Beyond Endoscopy} This line of inquiry has been extended using \emph{trace formulae}. Initiated by Langlands \cite{Lan04}, the Beyond Endoscopy program aims to address his general functoriality conjecture by the poles of automorphic $L$-functions and a novel comparison of trace formulae. The first example was due to Venkatesh \cite{Ven04}. He showed that the pole of the symmetric square $L$-function for $\hbox{GL}(2)$ at $s=1$ detects the automorphic induction of Hecke Gr\"{o}ssencharakters to dihedral forms of $\hbox{GL}(2)$.

The ideas of Beyond Endoscopy have been adapted to various contexts, providing new proofs of several classical results. For example, Altu\u{g} \cite{Alt-2} gave a new proof of the subconvex bound $ a_{f}(n) \, = \,  O_{f,\, \epsilon}(n^{1/4+\epsilon})$ for the normalized Fourier coefficients of a cusp form $f$ of $\hbox{GL}(2)$, originally due to Selberg and Kuznetsov. Recent analytic applications of Beyond Endoscopy include  \cite{RanSelBE, Qi21, BFL23, BL24, emory2024BE}.


In this article, we study the \emph{analytic continuation} and \emph{functional equations} of $L$-functions through the method of Beyond Endoscopy. In the spirit of Venkatesh's thesis (cf. \cite[Thm. 1]{Ven-thesis} and \cite[Prop. 2]{Ven04}), it is both interesting and important to understand the sources of arithmetic invariants from the geometric sides of trace formulae. This pertains to the \emph{root numbers} in our context with the level structures of the cusp forms being essential considerations; see Sect. \ref{sketchoA}.


\subsection{Set-up}

Throughout this work, $D\ge 1$ and $k\ge 4$ are integers, and $\chi$ is a Dirichlet character $(\bmod\, D)$. Denote by $S_{k}(D, \chi)$ the space of holomorphic cusp forms of weight $k$, level $D$ and nebentypus $\chi$, equipped with the Petersson inner product:
\begin{align}\label{Petinn}
    ||f||^2 \ := \ \int_{\Gamma_{0}(D) \setminus \mathbb{H}} \ |f(z)|^2 y^{k} \ \frac{dx dy}{y^2} \hspace{20pt} (f \, \in \, S_{k}(D, \chi)),
\end{align}
where  $z\in \mathbb{H}$.  We impose $\chi(-1) = (-1)^k$.
 For $f\in S_{k}(D,\chi)$, the Fourier expansion 
 \begin{align}\label{conven}
	f(z) \ &= \  \sum_{n=1}^{\infty} \ a_{f}(n)n^{\frac{k-1}{2}} e(nz) \hspace{15pt} (z\in \mathbb{H})
\end{align}
holds, where $e(z):= e^{2\pi i z}$.   The standard $L$-function $\mathcal{L}(s,f)$ and its dual $\overline{\mathcal{L}}\left(s, f\right)$ for $f\in S_{k}(D, \chi)$ are defined by the Dirichlet series \footnote{ Except for the notation ``$\overline{\mathcal{L}}$'', all other bar notations in this article denote the complex conjugation.  }
\begin{align}\label{holoDS}
	\mathcal{L}(s,f)  :=  \sum_{n=1}^{\infty}  \frac{a_{f}(n)}{n^s} \hspace{10pt} \text{ and } \hspace{10pt} \overline{\mathcal{L}}\left(s, f\right) :=  \overline{\mathcal{L}(\overline{s},f)}  =  \sum_{n=1}^{\infty}  \frac{\overline{a_{f}}(n)}{n^s} \hspace{10pt} (\re s \gg 1).
\end{align}

We identified a minimal set of conditions on $a_{f}(n)$ such that the functional equation takes the form $\mathcal{L}(s,f) \,\leftrightarrow\, \overline{\mathcal{L}}(1-s,f)$ as specified by Selberg's axioms:
\begin{assu}\label{mainHecput}
There is an orthogonal basis $\mathcal{B}_{k}^*(D, \chi)$ of $S_{k}(D, \chi)$ such that for any $f\in \mathcal{B}_{k}^*(D, \chi)$,  the Fourier coefficients of $f$ satisfy:
    \begin{enumerate}
       \item \label{adjoint} $\overline{a_{f}(\ell)} =a_{f}(\ell) \overline{\chi}(\ell)$ whenever  $(\ell, D)=1$, and
       
        \item \label{multas} $a_{f}(nm)=a_{f}(n)a_{f}(m)$ whenever $n\ge 1$ and $m\mid D^{\infty}$. 
    \end{enumerate}
\end{assu}

In the classical theory of modular forms, dating back to Hecke and Petersson, there are plenty of examples where Assumption \ref{mainHecput} is satisfied; see Sect. \ref{assum}.  Also, Assumption \ref{mainHecput}.(\ref{adjoint})--(\ref{multas}) are weaker than requiring $\mathcal{L}(s,f)$ to admit the standard Euler product.

\subsection{Main results}
 Using the Petersson trace formula (Lem. \ref{Peterform}), we present a new proof of the analytic continuation and functional equation for $\mathcal{L}(s,f)$ by working directly with its definition in (\ref{holoDS}), i.e., without relying on any integral representations of $L$-functions.  This is distinct from Hecke's approach.

\subsubsection{Primitive nebentypus}

Our first result is an averaged version of the Voronoi formula. Its proof, \emph{independent} of Assumption \ref{mainHecput}, serves as the main component of our beyond-endoscopic implementation.

\begin{thm}\label{FourVorthm}
Let $D>1$, $k\ge 4$ and $\ell\ge 1$ are integers, and $\chi\, (\bmod\, D)$ is a primitive character satisfying $\chi(-1)=(-1)^k$. Write $\ell= \ell_{0}\ell'$, where $\ell_{0}:=(\ell, D^{\infty})$ and $(\ell', D)=1$. \footnote{ The notations $\ell_{0}$, $\ell'$ will be used across this text with the same meaning. } Then
\begin{align*}
\sideset{}{^h}{\sum}_{f \in \mathcal{B}_{k}(D, \chi)} \, &\overline{a_{f}(\ell)} \  \sum_{n=1}^{\infty} \,    a_{f}(n) g(n)\nonumber\\
     &\hspace{-10pt} =  \frac{ i^{k}\chi(-1) \epsilon_{\chi}}{D^{1/2}}   \ \  \sideset{}{^h}{\sum}_{f \in \mathcal{B}_{k}(D, \chi)} \, a_{f}\left(\ell'\right)\overline{\chi}(\ell') \sum_{n=1}^{\infty}  \, \overline{a_{f}(\ell_{0}Dn)} \,  \left( \, 2\pi\int_{0}^{\infty} \ g(x) J_{k-1}\left(4\pi\sqrt{\frac{nx}{D}}\right) dx\right) 
\end{align*}
   for any $g\in C_{c}^{\infty}(0, \infty)$ and orthogonal basis  $\mathcal{B}_{k}(D, \chi)$ of\, $S_{k}(D, \chi)$, where 
   \begin{itemize}
       \item $\epsilon_{\chi}$ is the Gauss sum associated with $\chi$, defined by \, $\epsilon_\chi  :=  D^{-1/2} \sum_{\alpha\Mod{D}}\chi(\alpha)e\left(\alpha/D\right);$
   
       \item $J_{k-1}(\, \cdot \,)$ is the $J$-Bessel function; 
 
       \item the harmonic average is given by $\sideset{}{^h}{\sum}_{f \in \mathcal{B}_{k}(D, \chi)} \ \alpha_{f} \ := \  \frac{\Gamma(k-1)}{(4\pi)^{k-1}} \,  \sum_{f \in \mathcal{B}_{k}(D, \chi)} \, \frac{\alpha_{f}}{||f||^2}. $
 
   \end{itemize}    
\end{thm}


From this, we derive the functional equation for $\mathcal{L}(s,f)$ under Assumption \ref{mainHecput} \footnote{which in fact always holds for primitive nebentypus. }.

\begin{thm}\label{axiomFE}
Suppose $D>1$, $k\ge 4$ are integers,  and $\chi\, (\bmod\, D)$ is  primitive with $\chi(-1)=(-1)^k$. Under Assumption \ref{mainHecput}, the $L$-functions $\mathcal{L}(s,f)$ and $\overline{\mathcal{L}}(s,f)$ admit entire continuation and satisfy the functional equation 
\begin{align}\label{primnebFE}
    \mathcal{L}(s, f) \ = \  i^{k} \chi(-1) \epsilon_{\chi} \overline{a_{f}(D)} D^{\frac{1}{2}-s} \,  \frac{\gamma_{k}(1-s)}{\gamma_{k}(s)} \, \overline{\mathcal{L}}\left(1-s, f\right)
\end{align}
for any  $f\in \mathcal{B}_{k}^*(D, \chi)$ and $s\in \C$, where
\begin{align}\label{gamks}
	\gamma_{k}(s) \, := \,  c_{k}\,  (2\pi)^{-s} \, \Gamma\left(s+ \frac{k-1}{2}\right) \, =  \, \pi^{-s}\,  \Gamma\left(\frac{s+\frac{k-1}{2}}{2}\right) \Gamma\left(\frac{s+\frac{k+1}{2}}{2}\right), \hspace{10pt} c_{k} \, := \, 2^{(3-k)/2} \sqrt{\pi}. 
\end{align}
    
\end{thm}


\subsubsection{Trivial nebentypus}

We obtain analogous results when  $\chi=\chi_{0} \, (\bmod\, D)$ is trivial, but $D$ is required to be square-free. Assumption \ref{mainHecput} is non-vacuous (see Sect. \ref{assum}).

\begin{thm}\label{trivthm}
Suppose $k\ge 4$ and $D\ge 1$ are integers with $k$ even and $D$ square-free.  
\begin{enumerate}
    \item \label{avgFourVortriv} For any $g\in C_{c}^{\infty}(0,\infty)$ and any orthogonal basis of $S_{k}(D)$, \footnote{  We write $S_{k}(D):=S_{k}(D, \chi_{0})$ and $\mathcal{B}_{k}^*(D)= \mathcal{B}_{k}^*(D,\chi_{0})$. 
} we have
\begin{align*}
\hspace{30pt} \sideset{}{^h}{\sum}_{f \in \mathcal{B}_{k}(D)} \, &\overline{a_{f}(\ell)} \  \sum_{n=1}^{\infty} \,    a_{f}(n) g(n) \nonumber\\
 &=  \mu(D)i^{k}   \ \  \sideset{}{^h}{\sum}_{f \in \mathcal{B}_{k}(D)}  a_{f}\left(\ell'\right)\sum_{n=1}^{\infty}  \, \overline{a_{f}(\ell_{0}Dn)} \,  \left( 2\pi\int_{0}^{\infty} g(x) J_{k-1}\left(4\pi\sqrt{\frac{nx}{D}}\right)  dx\right). 
\end{align*}

\item \label{triFEpt} If $D>1$ and Assumption \ref{mainHecput} also hold, then the $L$-functions $\mathcal{L}(s,f)$ and $\overline{\mathcal{L}}(s,f)$ admit entire continuation,  and satisfy  the functional equation 
\begin{align}\label{trivFE}
    \mathcal{L}(s, f) =   i^{k} \sqrt{D} \, \mu(D) \,   \overline{a_{f}(D)}\,   D^{\frac{1}{2}-s} \,  \frac{\gamma_{k}(1-s)}{\gamma_{k}(s)} \,\overline{\mathcal{L}}\left(1-s, f\right)
\end{align}
for any $f\in \mathcal{B}_{k}^*(D)$ and $s\in \C$, where $\mu(\cdot)$ is the M\"obius $\mu$-function. 

\item \label{ptD1ACFE} If $D=1$ and $k$ as above, then  for any  $f\in S_{k}(1)$ and $s\in \C$, the $L$-function $\mathcal{L}(s,f)$ admits an entire continuation, and satisfy the functional equation 
 \begin{align}\label{level1FE}
      \mathcal{L}(s, f) =   i^{k}  \,  \frac{\gamma_{k}(1-s)}{\gamma_{k}(s)} \, \mathcal{L}(1-s,f). 
 \end{align}

\end{enumerate}

\end{thm}



\subsection{Sketch of argument}\label{sketchoA}

 We follow the overarching principle of  \emph{Spectral Reciprocity}; see \cite{BK-duke}.   Analytically, special care is necessary regarding the convergence of sums and integrals. Arithmetically, we study how transformations of the character sums lead to the occurrence of root numbers. In short, one must compare the twisted and untwisted Kloosterman sums with different conductors: 
\begin{align}\label{comparow}
    S_{\chi}(\ell, n; cD) \ \ \longleftrightarrow \ \ \overline{\chi}(c) S(n, \ell\overline{D};c) \ \  \longleftrightarrow \ \ S_{\chi}\left(\ell_{0}Dn, \ell'; cD\right),
\end{align}
where $(c,D)=1$, $D\overline{D} \equiv 1\, (\bmod\, c)$, and $\ell=\ell_{0}\ell'$ with $\ell_{0}:=(\ell, D^{\infty})$.

\subsubsection{Local ingredients} The first comparison of (\ref{comparow}) follows from the first four steps of Sect. \ref{mainbody}, along with an analytic-arithmetic cancellation involving the Hankel inversion and  Weber's identity (Lem. \ref{hankelinv}--\ref{keyspec}), which serve as the archimedean inputs. 

The second comparison of (\ref{comparow}) follows from the \emph{twisted multiplicativity} of Kloosterman sums, which serves as a key non-archimedean ingredient; see (\ref{twisnotw}) and (\ref{trivtwmu}). The arithmetic factor of this comparison is given by the  \emph{Gauss sum} when $\chi$ is primitive and the \emph{Ramanujan sum} when $\chi$ is trivial, accounting for the shapes of the root numbers in Thm. \ref{axiomFE}--\ref{trivthm}. Also, it is crucial to observe the \emph{vanishing} of the last  Kloosterman sum in (\ref{comparow}) when $(c,D)>1$ (Lem. \ref{basictwisKL}), and the \emph{non-vanishing} for the Ramanujan sum when $D$ is square-free, which play subtle but important roles in the final steps of Sect. \ref{mainbody} and \ref{sqfree}.

\subsubsection{Global ingredients} The automorphic (global) inputs to our theorems are fairly modest, ultimately relying only  
on the fact that $S_{k}(D,\chi)$ is a finite-dimensional inner product space and that each $f \in S_{k}(D,\chi)$ has a Fourier expansion.

\subsubsection{Trace formula} We make use of the Petersson trace formula rather than the Eichler--Selberg trace formula in Altu\u{g} \cite{Alt-3}. The main theorem of \cite{Alt-3} is that $L(s,\,\Delta)$  admits an analytic continuation to the region $\re s>31/32$ for Ramanujan's $\Delta$-function. While refining his intricate analysis could extend this region, substantial new inputs seem essential to reach $\re s> 1/2- \delta$, which is necessary for discussing the functional equations. Indeed,  the average of $a_{\Delta}(n)$ considered in \cite{Alt-3} uses the \emph{sharp cut-off} (see p. 1357 therein), and the fact  $S_{12}(1) =\C \cdot \Delta$ was needed to relate this average to his beyond-endoscopic average. These justify our use of the Petersson formula.


\subsection{Comments on the assumptions}\label{assum}

Assumption \ref{mainHecput}.(\ref{adjoint}) holds provided the orthogonal basis $\mathcal{B}_{k}^{*}(D, \chi)$ consists of simultaneous eigenforms for Hecke operators  $T_{\ell}$ with $(\ell, D)=1$.  Such a basis exists and we have $T_{\ell}^{*} = \overline{\chi}(\ell) T_{\ell}$ for $(\ell, D)=1$, where $T_{\ell}^{*}$ is the adjoint with respect to the Petersson inner product (\ref{Petinn}). These were known to Petersson.

If $\mathcal{B}_{k}^{*}(D, \chi)$ consists of simultaneous eigenforms for all Hecke operators  $T_{\ell}$ with $\ell \ge 1$, Hecke showed that Assumption \ref{mainHecput}.(\ref{multas}) holds. The existence of such a basis is less obvious but was also due to Hecke (see \cite[Chp. 14]{IK}) in special cases:
\begin{itemize}
    \item The nebentypus $\chi \, (\bmod\, D)$ is primitive with $D>1$ and $\chi(-1)=(-1)^k$;

    \item The nebentypus $\chi \, (\bmod\, D)$ is trivial with $D>1$ being prime, and the weight $k\in\{2,4,6,8,10,14\}$; 

    \item $D=1$ and $k\ge 12$ is even. 
\end{itemize}

The functional equation (\ref{trivFE}) remains valid for any even $k\ge 2$, square-free $D\ge 1$ and newform $f\in S_{k}(D)$ (\cite[Prop. 14.16]{IK}).  It should be possible to prove this using Beyond Endoscopy but with the refined Petersson formula; see, e.g., \cite[Thm. 3]{PY-MathAnn}.

\subsection{A related work}

We are aware of the earlier work by Herman \cite{HermanBeFe}, which, as pointed out in \cite[Sect. 6.3]{Sak-IHES}, has served as a prototype of several important research directions in Beyond Endoscopy. However, our careful examination revealed significant errors in \cite{HermanBeFe}. We believe a self-contained write-up that fully rectifies these issues is of interest. Furthermore, our argument is more streamlined, and certain observations presented in Sect. \ref{sketchoA} and the main body of this article do not appear to be present in the literature. The major errors of \cite{HermanBeFe}
are discussed as follows. 
 \begin{enumerate}
    \item \label{misrecip} The reciprocity relation alluded to  (\ref{comparow}) is crucial to \cite{HermanBeFe} and this article. However, the corresponding relation in \cite{HermanBeFe} is incorrect, partly due to misuse of the additive reciprocity (\cite[(3-5)]{HermanBeFe}). (Compare with our (\ref{compareHer}).)  

    \item \label{onlyunr} The additional restriction $(\ell,D)=1$ in \cite{HermanBeFe} would prevent him from deducing the functional equation of $\mathcal{L}(s,f)$. \footnote{ Also, the functional equation of $\mathcal{L}(s,f)$ was incorrectly stated a couple of times in \cite{HermanBeFe}!}

     \item \label{absone} Herman's proof of the functional equation (\ref{primnebFE}) crucially relied on the property $|a_{f_{D}}(D)|=1$ (\cite[pp. 508-510]{HermanBeFe}), \footnote{ The object `$f_{D}$' was not properly defined in \cite{HermanBeFe}, but its meaning may be inferred from \cite[line -7, p. 502]{HermanBeFe}.}  which is typically viewed as a consequence of determining the root number. Thus, assuming $|a_{f_{D}}(D)|=1$ in the beyond-endoscopic approach to (\ref{primnebFE}) would defeat its purpose. This property was neither proved in \cite{HermanBeFe} nor does it seem provable from the trace formula alone.

     \item \label{nonvan}  The assumption $a_{f}(D)\neq 0$ is needed for the well-definedness of \cite[(3-1)]{HermanBeFe}.

     \item \label{anassu} Based on the discussion in \cite[Appendix]{HermanBeFe}, an essential technical assumption on the growth of spectral parameters seems to be missing. 

     \item \label{Polygro} The convergence of the key identity \cite[(4-1)]{HermanBeFe} is unclear from his argument. 
 \end{enumerate}

We circumvent these concerns with a different argument, without relying on $|a_{f_{D}}(D)|=1$ or $a_{f}(D)\neq 0$, and our Thm. \ref{FourVorthm} holds for all $\ell \ge 1$. We also 
eliminate the requirement for the level $D$ to be square-free, as imposed in \cite{HermanBeFe}, for the case of primitive nebentypus. Our results for the case of trivial nebentypus are new.   Concerns (\ref{onlyunr}) and   (\ref{anassu})--(\ref{Polygro}) above reflect the algebraic and analytic subtleties, respectively, in deducing the functional equation of $\mathcal{L}(s,f)$ from the averaged Voronoi formula, which were left unaddressed in \cite{HermanBeFe}.  We remove the growth assumption of (\ref{anassu}) by carefully analyzing an oscillatory integral, and for (\ref{Polygro}), we ensure the polynomial growth of 
$\mathcal{L}(s,f)$ using only the trace formula. See Sect. \ref{ACGproof} for details. Regarding (\ref{onlyunr}), the discussion in \cite[Sect. 4]{HermanBeFe} was incomplete. In Sect. \ref{deduceFE}, we apply ideas from Venkatesh's thesis \cite[Sect. 2.6 and  3.2]{Ven-thesis} instead and carefully work through the steps, it becomes clear that Assumption \ref{mainHecput} is necessary---though it was unfortunately overlooked in \cite{HermanBeFe}.

The proof of Thm. \ref{FourVorthm} requires careful analysis. We found that $k\ge 4$ is necessary to ensure the validity of the argument,  whereas \cite{HermanBeFe} claimed $k\ge 2$ is sufficient without justification. Handling the cases $k\in \{2,3\}$ requires the use of  ``Hecke's trick''.

We also found a number of other gaps and issues in \cite{HermanBeFe}. To the best of our understanding, these are significant. Along with our major concerns listed above,  his final geometric expansion (\cite[(3-17)]{HermanBeFe}) is incorrect and does not match with our (\ref{finalgeo}). See our online note \cite{Erra} for full discussions.

\subsection{Acknowledgement}
The research is supported by the EPSRC grant: EP/W009838/1. We thank Matthew P. Young for his comments and suggestions.


\section{Preliminaries}\label{prelim}

\subsection{Integral transforms and Bessel functions}\label{BesstransSec}
Let $\phi \in C_{c}^{\infty}(\R)$ and\,  $\psi \in C_{c}^{\infty}(0,\infty)$. The \emph{Fourier transform} of $\phi$ and the \emph{Mellin transform} of $\psi$ are respectively given by 
\begin{align*}
     \widehat{\phi}(y) := \int_{\R} \ \phi(x) e(-xy)  dx \hspace{10pt} (y  \in  \R) \hspace{10pt} \text{ and } \hspace{10pt} 
     \widetilde{\psi}(s) :=   \int_{0}^{\infty}  \psi(x) x^{s-1}  dx \hspace{10pt} (s  \in   \C).
\end{align*}
Their respective inversion formula,  which holds provided the integral converges absolutely, is given by:
\begin{align}\label{melT}
    \phi(x) =  \int_{\R}   \widehat{\phi}(y) e(xy) dy =:  \check{\widehat{\phi}}(x)  \hspace{10pt} \text{ and } \hspace{10pt}     \psi(x)  = \int_{(\sigma)}  \widetilde{\psi}(s) x^{-s}  \frac{ds}{2\pi i}.
   \end{align}
   
 We record a useful estimation of oscillatory integrals from \cite[Lem. 8.1]{BKY}.
\begin{lemma}\label{BKYibp}
    Let $h\in C^{\infty}[\alpha, \beta]$ be a real-valued function and $w\in C_{c}^{\infty}[\alpha, \beta]$. Suppose $W, V, H,G, R>0$ are parameters such that the following bounds hold for any  $t\in [\alpha, \beta]$:
\begin{itemize}
    \item $w^{(j)}(t) \, \ll_{j} \,  W/V^{j}$ for any $j\ge 0$,

    \item $h^{(j)}(t) \ll_{j} H/G^{j}$ for any $j\ge 2$, and

    \item $|h'(t)| \, \ge \, R$. 
\end{itemize}
Then for any $A\ge  0$, we have
\begin{align}
    \int_{\R} \, w(t) e\left(h(t)\right) \ dt \ \ll_{A}  \  \left(\beta -\alpha\right) W \left(  \frac{1}{RV} \ + \  \frac{1}{RG}  \ + \  \frac{H}{(RG)^2}    \right)^A.  
\end{align}
 
\end{lemma}

  Recall the following integral representation for  $J$-Bessel functions (\cite[p. 192]{Wa95}): 
\begin{lemma}\label{MelBarJ}
    Let $k\ge 2$ and $\gamma_{k}(s)$ be defined as in (\ref{gamks}). Then we have \begin{align}\label{bessgamm}
 	J_{k-1}(4\pi x)  =     \frac{1}{2\pi} \, \int_{(\sigma)}  \frac{\gamma_{k}(1-s)}{\gamma_{k}(s)}  x^{2(s-1)} \, \frac{ds}{2\pi i}  \hspace{15pt} (x>0, \hspace{5pt}  1< \sigma<(k+1)/2). 
 \end{align}

\end{lemma}

 Let $F\in C_{c}^{\infty}(0, \infty)$ and $k\ge 1$ be an integer. The  \textit{Hankel transform} of $F$ is defined by
\begin{align}\label{Hatrasdef}
	(\mathcal{H}_{k}F)(a) &:=   2\pi \int_{0}^{\infty}  F(x) J_{k-1}(4\pi\sqrt{ax})  dx \hspace{15pt} (a  >   0).
\end{align}
The rapid decay of $(\mathcal{H}_{k}F)(a)$ as $a\to \infty$ can be seen by repeated integration by parts in (\ref{Hatrasdef}) using the equality (see \cite[p. 206]{Wa95}):
     \begin{align}\label{asympJbess}
        J_{k-1}\left( 2\pi x\right)  =  W_{k}(x) e(x)   +  \overline{W_{k}}(x) e(-x)
    \end{align}
    for $x>0$,  where $W_{k}$ is a smooth function satisfying
    \begin{align}\label{bessder}
         x^{j}(\partial^{j}W_{k})(x)  \ll_{j,k}   x^{-1/2}
    \end{align}
    for any $j\ge 0$ and $x>1$. As $a\to 0$, we have the bound 
     \begin{align}\label{basicHank}
        (\mathcal{H}_{k}F)(a)  \ll_{k}  a^{\frac{k-1}{2}}, 
    \end{align}
    which follows easily from the estimate
\begin{align}\label{Jbessunibd}
        J_{k-1}(y)  \ll_{k}   y^{k-1} \hspace{10pt} \text{ for } \hspace{10pt} y  >  0. 
    \end{align}

\begin{lemma}
For $k> 2$ and $a \in (0,1)$, we have
     \begin{align}\label{derHank}
        (\mathcal{H}_{k}F)'(a) \ \ll_{k} \ a^{\frac{k-3}{2}} \hspace{20pt} \text{ and } \hspace{20pt} (\mathcal{H}_{k}F)''(a) \ \ll_{k} \ a^{\frac{k-5}{2}}.
    \end{align}
\end{lemma}

\begin{proof}
Recall the recurrence relation
\begin{align}\label{recur}
        2 J_{k}'(z) \ =  \ J_{k-1}(z) - J_{k+1}(z).
    \end{align}
It then follows from (\ref{Jbessunibd}) that
\begin{align}
	(\mathcal{H}_{k}F)'(a) \ll_{k} a^{-1/2} \int_{x\asymp 1}  \left\{ (\sqrt{ax})^{k-2} +   (\sqrt{ax})^{k}\right\}  dx  \ll_{k} a^{\frac{k-3}{2}}. 
\end{align}
  The second bound of (\ref{derHank}) can be proved similarly. 
\end{proof}
    
The following two results can be found in \cite[Chp. 14.3-4, 13.31]{Wa95}.
    
\begin{lemma}[Hankel inversion formula]\label{hankelinv}
For any  $F\in C_{c}^{\infty}(0, \infty)$, we have
\begin{align}
	(\mathcal{H}_{k} \circ \mathcal{H}_{k}F)(b) \ = \  F(b) \hspace{20pt} (b \ > \  0). 
\end{align}
\end{lemma}

   \begin{lemma}\label{keyspec}
Let  $k \ge 2$, $\re \alpha>0$, and  $\beta, \gamma>0$. Then
\begin{align}\label{keyspecinst}
   \int_{0}^{\infty}  e^{-2\pi \alpha y} J_{k-1}(4\pi \beta\sqrt{y}) J_{k-1}(4\pi \gamma\sqrt{y}) dy  =   \frac{i^{1-k}}{2\pi \alpha}\,  J_{k-1}\left( \frac{4\pi i\beta\gamma}{\alpha}\right) \exp\left(-  \frac{2\pi(\beta^2+\gamma^2)}{\alpha}\right).  
\end{align}   

\end{lemma}

\subsection{Kloosterman sums}\label{Kloospreq}

Let $m,n,c, D\in \Z$ with $c, D \ge 1$ and $c\equiv 0 \ (\bmod\, D)$. For a Dirichlet character $\chi \, (\bmod\, D)$,  the twisted Kloosterman sum is defined by
\begin{align}\label{twisKloos}
 	S_{\chi}(m, n;c) \ := \ \sideset{}{^*}{\sum}_{x \, (\bmod\, c)} \ \overline{\chi}(x) e\left(\frac{  m\overline{x}+ nx}{c}\right).
 \end{align}
  When $\chi$ is trivial, (\ref{twisKloos}) becomes $S(m,n;c)= \sideset{}{^*}{\sum}_{x \, (\bmod\, c)}  e\left(\frac{  m\overline{x}+ nx}{c}\right)$. The following fact is likely well-known, but it is not easy to locate a reference.

 \begin{lemma}\label{basictwisKL}
 If there is a prime $p$ such that  $p\mid m$, $p\nmid n$ and  $p^2 \mid c$, then $S_{\chi}(m,n;c)  =   0$. 
	
\end{lemma}

\begin{proof}
By assumptions, we may write  $c=p^{\alpha} c'$ with $\alpha \ge 2$ and $p\nmid c'$. Then 
	\begin{align*}
		S_{\chi}(m,n;c) \ = \  S_{\chi^{(c')}}\left(\bar{p}^{\alpha}m, \bar{p}^{\alpha}n;  c'\right) S_{\chi^{(p^{\alpha})}}(\overline{c'}m, \overline{c'}n; p^{\alpha}). 
	\end{align*}
	 Set $\psi= \chi^{(p^a)}$, $m_{1}=\overline{c'}m$ and $n_{1}=\overline{c'}n$. It suffices to show that
  \begin{align}\label{factwis}
      S_{\psi}(m_{1}, n_{1}; p^{\alpha}) \ &=  \  \sideset{}{^*}{\sum}_{x \, (\bmod\, p^{\alpha})} \overline{\psi}(x) e\left(\frac{m_{1}\overline{x}+n_{1} x}{p^{\alpha}}\right) \ = \ 0,
  \end{align}
  where $m_{1}=pm_{2}$, $m_{2} \in \Z$ and $p\nmid n_{1}$. Write $x=y  p^{\alpha-1} +z$ with $y \,(p) $ and $z\,  (p^{\alpha-1})$. Then
	\begin{align}\label{breakup}
		S_{\psi}(m_{1}, n_{1}; p^{\alpha})
		\ &=  \  \sideset{}{^*}{\sum}_{z \, (\bmod\, p^{\alpha-1})} \ e\left(\frac{m_{2}\overline{z}}{p^{\alpha-1}}\right) e\left(\frac{n_{1}z}{p^{\alpha}}\right) \overline{\psi}(z)
		\   \sum_{y (\bmod\, p)} \ e\left(\frac{n_{1}y}{p}\right) \overline{\psi}(1+\bar{z}yp^{\alpha-1}), 
	\end{align}
	  Note that $w\mapsto \overline{\psi}(1+wp^{\alpha-1})$ is an additive character $(\bmod\, p^{\alpha-1})$ when $\alpha\ge 2$. So, there exists $B\in \Z$ such that $\overline{\psi}(1+wp^{\alpha-1})= e(Bw/p^{\alpha-1})$ for any $w\in \Z$. In particular, $\overline{\psi}(1+\bar{z}yp^{\alpha-1})= e(B\bar{z}y)=1$. Now, the $y$-sum of (\ref{breakup}) vanishes as $p\nmid n_{1}$. 
\end{proof}


\subsection{Summation formulae}

     \begin{lemma}[Poisson summation]\label{lemPois}
Let $c, X>0$ and $c\in \Z$. Let  $V\in C_{c}^{\infty}(\R)$ and $K: \Z\rightarrow \mathbb{C}$ be $c$--periodic.  Then
	\begin{align}
			\sum_{n\in \Z}   K(n)V(n/X)  =   \frac{X}{c} \sum_{m\in \Z}  \,  \bigg( \, \sum_{\gamma (c)} \, K(\gamma)  e_c\left(m\gamma\right)\bigg)    \int_{0}^{\infty}  V(y)\, e\left(-\frac{mXy}{c}\right) dy. 
	\end{align}
\end{lemma}

 \begin{lemma}[Petersson trace formula]\label{Peterform}
Let $D\ge 1$ and $k\ge 4$ be integers, and  $\chi \, (\bmod\, D)$ be a Dirichlet character. For any $m, n\ge 1$, 
      \begin{align}\label{Petneb}
 		\sideset{}{^h}{\sum}_{f \in \mathcal{B}_{k}(D, \chi)} \ \overline{a_{f}(m)} a_{f}(n)   =  \delta(m=n) + 2\pi i^{-k}  \sum_{\substack{c>0 \\ c\equiv 0\, (D) }}  \frac{S_{\chi}(m,n;c)}{c} J_{k-1}\left(\frac{4\pi\sqrt{mn}}{c}\right).
 \end{align}
 
 \end{lemma}


\section{Proof of Theorem \ref{FourVorthm}}\label{mainbody}

Suppose $\ell, D, k$ are integers such that $\ell\ge 1$, $D>1$ and $k\ge 4$. Suppose $\chi\, (\bmod\, D)$ is a  Dirichlet character such that $\chi(-1)=(-1)^k$. Let $g\in C_{c}^{\infty}(0,\infty)$ be fixed. We define
\begin{align}\label{truncIsum}
I_k(\ell; D, \chi)  \ := \  \sum_{n=1}^{\infty} \, g(n)  \sideset{}{^h}{\sum}_{f \in \mathcal{B}_{k}(D, \chi)} \, \overline{a_{f}(\ell)} a_{f}(n).  
\end{align} 


\subsection{Step 1: Petersson--Poisson}\label{Petersect}

Apply the Petersson formula, open up the Kloosterman sum by its definition, and rearrange the summations, (\ref{truncIsum}) is given by
\begin{align}
I_k(\ell; D, \chi) \ = \  g(\ell) \ + \  2\pi i^{-k} \sum_{c=1}^{\infty} \, (cD)^{-1}  \sideset{}{^*}{\sum}_{ x \, (cD)} \,   \overline{\chi}(x) e\left(\frac{\ell \overline{x}}{cD}\right)  \sum_{n\in \mathbb{Z}} \, g(n) J_{k-1}\left(\frac{4\pi\sqrt{\ell n}}{cD}\right)  e\left(\frac{nx}{cD}\right). \nonumber
\end{align}
It follows from Poisson summation $(\bmod\, cD)$ to the $n$-sum that  
\begin{align}\label{evagammsum}
	I_k(\ell; D, \chi) =   g(\ell)  +    2\pi i^{-k}  \sum_{c=1}^{\infty} (cD)^{-1} \sideset{}{^*}{\sum}_{ x \, (cD)}   \overline{\chi}(x) e\left(\frac{\ell \overline{x}}{cD}\right)    \sum_{\substack{m\in \mathbb{Z}  \\ m\equiv -x \, (cD)}}  \int_{\R}  g(y)  J_{k-1}\left(\frac{4\pi\sqrt{\ell y}}{cD}\right)  e\left(-\frac{my}{cD}\right) dy.
\end{align}

When $m=0$,  we have  $x \equiv 0 \  (cD)$ and $(x,cD)=1$ in (\ref{evagammsum}),  which imply $cD =1$. This is impossible if $D>1$. In other words,  the dual zeroth frequency is \emph{absent} in this case and
\begin{align}\label{resulPetPo}
		I_k(\ell; D, \chi)   =   g(\ell)  +    2\pi i^{-k} \chi(-1) \, \sum_{c=1}^{\infty} \sum_{\substack{m\neq 0 \\ (m, cD)=1}}    \frac{\overline{\chi}(m)}{cD} e\left(-\frac{\ell \overline{m}}{cD}\right)  \int_{\R}  g(y)  J_{k-1}\left(\frac{4\pi\sqrt{\ell y}}{cD}\right)  e\left(-\frac{my}{cD}\right) \, dy.
\end{align}

\subsection{Step 2: Analytic--arithmetic reicprocity}\label{sect:step2}

With the rapid decay of $(\mathcal{H}_{k}g)(x)$ as $x \to 0$ and $x\to\infty$, it follows from Lem. \ref{hankelinv}--\ref{keyspec} and dominated convergence theorem that
\begin{align}
	 \int_{\R}  g(y)  J_{k-1}\left(\frac{4\pi\sqrt{\ell y}}{cD}\right)  e\left(-\frac{my}{cD}\right) \, dy  
	=  \frac{cD}{i^{k}m}  e\left( \frac{\ell}{cDm}\right)  \int_{0}^{\infty}\ (\mathcal{H}_{k}g)(x)   J_{k-1}\left( \frac{4\pi\sqrt{\ell x}}{m}\right) e\left(\frac{cDx}{m}\right) dx.  
\end{align}
Plugging the last expression into (\ref{resulPetPo}) and  applying the  additive reciprocity $\frac{\overline{m}}{cD}  +  \frac{\overline{cD}}{m}   \equiv   \frac{1}{cDm} \,   (\bmod\, 1) $
to the phase $e\left(-\ell \overline{m}/ cD\right)$ in (\ref{resulPetPo}), it follows that  
\begin{align}
	I_k(\ell; D, \chi)  =  g(\ell)  +    2\pi i^{-2k} \chi(-1) \sum_{c=1}^{\infty}  \sum_{\substack{m\neq 0 \\ (m, cD)=1}}     \frac{\overline{\chi}(m)}{m}  e\left(\frac{\ell \, \overline{cD}}{m}\right)    \int_{0}^{\infty} (\mathcal{H}_{k}g)(x)   J_{k-1}\left( \frac{4\pi\sqrt{\ell x}}{m}\right) e\left(\frac{cDx}{m}\right)  dx. \label{compareHer}
\end{align}


\subsection{Step 3: Analytic preparation for Poisson}\label{poisprep}

Note: $i^{-2k}\chi(-1)=1$. Firstly, we have
\begin{align*}
I_k(\ell; D, \chi)   =    g(\ell)  + \sum_{\pm}   2\pi  \, \sum_{c=1}^{\infty} \, \sum_{\substack{m> 0 \\ (m, cD)=1}} \,    \frac{\overline{\chi}(\pm m)}{\pm m}  e\left(\frac{\ell \, \overline{cD}}{\pm m}\right)  \int_{0}^{\infty} (\mathcal{H}_{k}g)(x)   J_{k-1}\left( \frac{4\pi\sqrt{\ell x}}{\pm m}\right) e\left(\frac{cDx}{\pm m}\right) dx.  
\end{align*}
Secondly,  the change $c\to -c$, $J_{k-1}(-y)=(-1)^{k-1} J_{k-1}(y)$ and  $\chi(-1)=(-1)^k$ lead to
\begin{align}
	I_k(\ell; D, \chi) =   g(\ell)  +    2\pi   \sum_{c\neq 0}  \sum_{\substack{m> 0 \\ (m, cD)=1}}     \frac{\overline{\chi}(m)}{m}  e\left(\frac{\ell \, \overline{cD}}{m}\right)       \int_{0}^{\infty} (\mathcal{H}_{k}g)(x)   J_{k-1}\left( \frac{4\pi\sqrt{\ell x}}{m}\right) e\left(\frac{cDx}{m}\right)  dx. \nonumber
\end{align}
Thirdly, observe that if one takes $c=0$ in the $m$-sum above, then  such a sum consists of a single term ($m=1$) and it is given by $2\pi\,   \int_{0}^{\infty} (\mathcal{H}_{k}g)(x)   J_{k-1}( 4\pi\sqrt{\ell x})  dx  =   g(\ell)$
using Lem. \ref{hankelinv}.  In other words, the diagonal term from the Petersson formula completes the $c$-sum after the step of applying `reciprocity', i.e., 
\begin{align}\label{harmresu}
	I_k(\ell; D, \chi)  =    2\pi  \, \sum_{c\in \mathbb{Z}} \, \sum_{\substack{m> 0 \\ (m, cD)=1}}    \frac{\overline{\chi}(m)}{m}  e\left(\frac{\ell \, \overline{cD}}{m}\right)     \int_{0}^{\infty} (\mathcal{H}_{k}g)(x)   J_{k-1}\left( \frac{4\pi\sqrt{\ell x}}{m}\right) e\left(\frac{cDx}{m}\right) \ dx. 
\end{align}
Fourthly, we make the change of variables $x\to xm/D$:
\begin{align}\label{compforPois}
    I_k(\ell; D, \chi)  =   \frac{2\pi}{D}   \sum_{c\in \mathbb{Z}} \, \sum_{\substack{m> 0 \\ (m, cD)=1}}     \overline{\chi}(m)  e\left(\frac{\ell \, \overline{cD}}{m}\right) \int_{0}^{\infty} (\mathcal{H}_{k}g)\left( \frac{mx}{D}\right)    J_{k-1}\left( 4\pi\sqrt{\frac{\ell x}{mD}}\right) e\left(cx\right)  dx. 
\end{align}

We must check that the double sum of (\ref{harmresu}) converges absolutely. This is indeed the case provided that  $k\ge 4$. Integrating by parts twice, for $c\neq 0$, the $x$-integral of (\ref{harmresu}) is
\begin{align}\label{IBPtwi}
  \, \ll \,   \left(\frac{cD}{m}\right)^{-2}   \sum_{i=0}^{2} \, \int_{0}^{\infty} \bigg| \partial_{x}^{2-i} (\mathcal{H}_{k}g)(x) \, \partial_{x}^{i} \left(J_{k-1}\left( \frac{4\pi\sqrt{\ell x}}{m}\right)\right)\bigg|  dx.
\end{align}

In the following, we appeal to the rapid decay of $\mathcal{H}_{k}g$ and (\ref{basicHank})--(\ref{recur}). When $i=0$, the integral  of (\ref{IBPtwi}) is $\ll_{k,\ell} m^{1-k}  \{ 1  +  \int_{0}^{1} \, x^{\frac{k-5}{2}}  x^{\frac{k-1}{2}} dx \}  \ll_{k} m^{1-k}$. For $x>0$, 
observe that $\partial_{x} \big(J_{k-1}\big( \frac{4\pi\sqrt{\ell x}}{m}\big)\big) 
      \ll_{k, \ell} m^{1-k}\max\{ x^{\frac{k-3}{2}}, x^{\frac{k-1}{2}}\} $. Therefore,  the integral  of (\ref{IBPtwi}) when $i=1$ is  $ \ll_{k, \ell}   m^{1-k}\{ 1  + \int_{0}^{1}  x^{\frac{k-3}{2}} x^{\frac{k-3}{2}} dx\}  \ll  m^{1-k}$.
Similarly, we have $ \partial_{x}^2 (J_{k-1}( \frac{4\pi\sqrt{\ell x}}{m}))  \ll_{k, \ell}  x^{\frac{k-5}{2}}/m^{k-1}$
for $0<x<1$ and the integral of (\ref{IBPtwi}) when $i=2$ is  $\ll_{k, \ell} m^{1-k}$. For $k\ge 4$, we conclude that the double sum of (\ref{harmresu})   converges absolutely.


\subsection{Step 4: Poisson}\label{secPois}

Let $\epsilon>0$ be given and  $h_{\epsilon}$ be a smooth function on $\R$ such that $h_{\epsilon}\equiv 1$ on $[\epsilon, \infty)$, $h_{\epsilon}\equiv 0$ on $(-\infty, 0]$, and $0\le h_{\epsilon}\le 1$ on $(0, \epsilon)$. Define $F_{\epsilon}: \R \rightarrow \C$ by 
\begin{align*}
	F_{\epsilon}(x)  :=  (\mathcal{H}_{k}g)(mx/D)   J_{k-1}( 4\pi\sqrt{\ell x/mD}) h_{\epsilon}(x).   
\end{align*}
Using the decay of $J_{k-1}(y)$ as $y\to 0$, observe that (\ref{harmresu}) can be rewritten as
\begin{align}
	I_k(\ell; D, \chi)  =    \frac{2\pi}{D}  \,  \sum_{\substack{m> 0 }} \,   \overline{\chi}(m) \, \cdot \,  \lim_{\epsilon\to 0+} \,  \sum_{\substack{ (c, m)=1}}  e\big(\frac{\ell \, \overline{cD}}{m}\big)   \check{(F_{\epsilon})}(c)
\end{align}
upon interchanging the order of the $c$-sum and $m$-sum. Poisson in $c$ $(\bmod\, m)$ gives 
\begin{align}\label{perturbPo}
		I_k(\ell; D, \chi)  \ &= \  \frac{2\pi}{D}  \,  \sum_{\substack{m> 0 }} \,   \overline{\chi}(m) \cdot \frac{1}{m} \, \lim_{\epsilon\to 0+} \, \sum_{n\in \mathbb{Z}} \, S(n, \ell\overline{D}; m) F_{\epsilon}(n/m).
\end{align}
By the rapid decay of $\mathcal{H}_{k}g$ and the bound $J_{k-1}(x)\ll_{k} x^{k-1}$, the expression on right side of  (\ref{perturbPo}) without  `$\lim_{\epsilon\to 0+}$'  is $ \ll_{k, A}  D^{-1} \sum_{m, n>0}   \left(\frac{n}{D}\right)^{-A}( \frac{1}{m}\sqrt{\frac{\ell n}{D}})^{k-1} \ \ll_{k, A, \ell, D} \  1$,
provided $k\ge 4$ and $A>(k+1)/2$.  Dominated convergence and Fubini's theorems give
\begin{align}\label{secoPoi}
  I_k(\ell; D, \chi)    \ &= \  2\pi  \,  \sum_{n>0} \, (\mathcal{H}_{k}g)\left(\frac{n}{D}\right)\, 
 \sum_{\substack{c> 0 }} \,   \frac{\overline{\chi}(c)}{cD}\, S(n, \ell\overline{D}; c)     J_{k-1}\left( \frac{4\pi \sqrt{\ell nD}}{cD}\right).
\end{align}


\subsection{Step 5: Twisted multiplicativity}\label{arithsum}

Suppose $(c,D)=1$. The reduced residue class $(\bmod\, cD)$ can be expressed as $c\overline{c} \alpha + D\overline{D} \beta$ with $\alpha$, $\beta$ over the reduced residue classes $(\bmod\, D)$,  $(\bmod\, c)$ respectively. 
Write $\ell= \ell_{0}\ell'$ with $\ell_{0}:=(\ell, D^{\infty})$ and $(\ell', D)=1$. Then
\begin{align}
    S_{\chi}(\ell_{0}Dn, \ell'; cD) 
    \ &= \ \sideset{}{^*}{\sum}_{\beta (c)} \ e\left( \frac{\ell_{0}n\, \overline{\beta}+\ell'\, \overline{D}\beta}{c}\right) \  \sideset{}{^*}{\sum}_{\alpha (D)}\, \overline{\chi}(\alpha) e\left(\frac{\ell' \, \overline{c}\alpha}{D}\right). \nonumber
\end{align}
Since $\ell_{0}\mid D^{\infty}$ and $(c,D)=1$ in (\ref{secoPoi}),  we have $(\ell_{0}, c)=1$. The change of variables $\beta \to \ell_{0}\beta$ and the \emph{primitivity} of $\chi\,  (\bmod\, D)$ imply that
\begin{align}\label{twisnotw}
     S_{\chi}(\ell_{0}Dn, \ell'; cD)  \  = \  S(n, \ell\overline{D}; c) \overline{\chi}(c) \chi(\ell') \epsilon_{\overline{\chi}}\sqrt{D}. 
\end{align}
It is essential to observe that (\ref{twisnotw}) remains valid even when $(c,D)>1$, by Lem. \ref{basictwisKL}.


\subsection{Step 6: Petersson in reverse}\label{Petinrev}

Since $D>1$ and $\chi\,  (\bmod\, D)$ is primitive, we have
\begin{align}
	I_k(\ell; D, \chi)  
	\ &= \  \frac{i^{k}\chi(-1) \overline{\chi}(\ell')\epsilon_{\chi}}{\sqrt{D}}   \, \sum_{n>0}  \, (\mathcal{H}_{k}g)\left(\frac{n}{D}\right) \, \cdot\, 2\pi i^{-k}  \sum_{\substack{c> 0 }} \,   \frac{  S_{\chi}(\ell_{0}Dn, \ell'; cD)}{cD} \,       J_{k-1}\left( \frac{4\pi \sqrt{\ell nD}}{cD}\right)\label{finalgeo}
\end{align}
from  $|\epsilon_{\chi}|=1$, (\ref{secoPoi}) and (\ref{twisnotw}).
Now, the Petersson formula can be readily applied to the innermost $c$-sums above and thus, 
\begin{align}
    I_k(\ell; D, \chi)  
	\ &= \ \frac{i^{k}\chi(-1) \overline{\chi}(\ell')\epsilon_{\chi}}{\sqrt{D}}  \ \  \sideset{}{^h}{\sum}_{f \in \mathcal{B}_{k}(D, \chi)} \, a_{f}\left(\ell'\right) \, \sum_{n>0}  \,  \overline{a_{f}(\ell_{0}Dn)} (\mathcal{H}_{k}g)\left(\frac{n}{D}\right).\nonumber
\end{align}
Notice that the diagonal contribution $\delta(\ell_{0} D n=\ell')$ from the Petersson formula is identically zero as $D>1$!  This completes the proof of Thm. \ref{FourVorthm}.


\section{Analytic continuation and polynomial growth}\label{ACGproof}

In this section, we show that $\mathcal{L}(s,f)$ can be analytically continued to the left of $\re s=1/2$ for any $f\in S_{k}(D, \chi)$, using (\ref{holoDS}) and (\ref{resulPetPo}). To deduce the functional equation of $\mathcal{L}(s,f)$, it is also necessary to establish \emph{polynomial growth} for
\begin{align}\label{avgACcont}
\mathcal{A}_{\ell}(s) :=   \sideset{}{^h}{\sum}_{f \in \mathcal{B}_{k}(D, \chi)}  \overline{a_{f}(\ell)} \mathcal{L}(s,f)
\end{align}
within the strip $0< \re s<1$ and as $|\im s|\to \infty$.

\begin{prop}\label{polygrAC}
Let  $\ell\ge 1$ and $k\ge 4$ be integers. Then  $\mathcal{A}_{\ell}(s)$ admits a holomorphic continuation to the half-plane \,$\sigma > -(k-4)/2$ and satisfies the estimate 
    \begin{align}\label{preast}
    \mathcal{A}_{\ell}(s) \  \ll_{k, \ell, D, \chi, \sigma} \ (1+|t|)^{k-2}
\end{align}
for any $s=\sigma+it$ with\,  $\sigma > -(k-4)/2 $ and $t\in \R$. 
\end{prop}

We begin our argument by inserting a smooth partition of unity. Let $g\in C_{c}^{\infty}[1,2]$ such that $1 =   \sum_{u\in \Z} \ g(x/2^u)$ for any $x>0$. Then for $\re s\gg  1$, we have
\begin{align*}
     \mathcal{L}(s,f) \, := \, \sum_{n=1}^{\infty} \, a_{f}(n)n^{-s} \, = \,  \sum_{u = -1}^{\infty}  \, \frac{1}{2^{us}} \, \sum_{n} \, a_{f}(n) G_{s}(n/2^u), \hspace{5pt} \text{where} \hspace{5pt} G_{s}(y) \, := \,  y^{-s} g(y). 
\end{align*}
Let 
\begin{align*}
    \mathcal{I}_{s}(X;\ell)  \, := \, \sideset{}{^h}{\sum}_{f \in \mathcal{B}_{k}(D, \chi)} \, \overline{a_{f}(\ell)} \, \sum_{n} \, a_{f}(n)G_{s}(n/X) \hspace{5pt} \text{ and } \hspace{5pt} \mathcal{A}_{\ell}(s)  :=   \sideset{}{^h}{\sum}_{f \in \mathcal{B}_{k}(D, \chi)} \, \overline{a_{f}(\ell)} \mathcal{L}(s,f)  =    \sum_{u=-1}^{\infty} \frac{\mathcal{I}_{s}(2^u;\ell)}{2^{us}}. 
\end{align*}
For  $X> 2\ell$, we have $G_{s}(\ell/X)=0$, and from  (\ref{resulPetPo}), it follows that \footnote{We have suppressed the less important dependence on $D,k, \chi$ in the notation $\mathcal{I}_{s}(X;\ell)$. }
\begin{align*}
  \mathcal{I}_{s}(X;\ell) =\frac{2\pi i^{-k} \chi(-1)}{D}   \sum_{c=1}^{\infty} \,  \sum_{\substack{m\neq 0 \\ (m, cD)=1}} \  \overline{\chi}(m) e\left(-\frac{\ell \overline{m}}{cD}\right)\,  \frac{X}{c}\,  \int_{\R}  G_{s}(y)  J_{k-1}\left(\frac{4\pi\sqrt{\ell yX}}{cD}\right)  e\left(-\frac{myX}{cD}\right)  dy,  
\end{align*}
\begin{align}\label{boundabs}
    \mathcal{I}_{s}(X; \ell) \ \ll_{ D} \  \, \sum_{c=1}^{\infty} \,  \ \sum_{m\neq 0} \ \frac{X}{c} \left| \, \int_{\R} \, G_{s}(y)  J_{k-1}\left(\frac{4\pi\sqrt{\ell yX}}{cD}\right)  e\left(-\frac{myX}{cD}\right) \, dy\right|. 
\end{align}

\begin{lemma}\label{weakt-asp}
Let $k\ge 4$. For any $X> 2\ell$ and  $s= \sigma+it$ with\, $\sigma, t\in \R$. Then
    \begin{align}
         \mathcal{I}_{s}(X;\ell) \ \ll_{k, \ell, D, \sigma} \ \left(1+|t|\right)^{k-2} X^{-(k-4)/2}.\label{larbd} 
    \end{align}

\end{lemma}

\begin{proof}[Proof of Lemma \ref{weakt-asp}]

Split the $c$-sum of (\ref{boundabs}) into two parts, depending on $c>\sqrt{\ell X}/20D$ and  $c\le \sqrt{\ell X}/20D$, which are denoted by  $ \mathcal{I}_{s, >}(X;\ell)$ and $\mathcal{I}_{s, \le}(X; \ell)$ respectively.

Consider $ \mathcal{I}_{s, >}(X;\ell)$. By Lem. \ref{MelBarJ} and integration by parts, the summand of (\ref{boundabs}) is
\begin{align}
   \ll  \   \frac{X}{c} \, \left(\frac{\sqrt{\ell X}}{cD}\right)^{2(a-1)}  \left(\frac{|m|X}{cD}\right)^{-r} \, \int_{(a)} \ \left|\frac{\gamma_{k}(1-v)}{\gamma_{k}(v)}  \right| \, \int_{\R} \ \left|\frac{d^r}{dy^r} \left[G_{s}(y)y^{v-1}\right]\right| \,  dy \  |dv| \nonumber
\end{align}
for $1< a<(k+1)/2$. Let $r\ge 2$,  $v= a +i\tau$ and $s= \sigma+it$ with $\sigma, t, \tau \in \R$. Then 
\begin{align}\label{ratiostir}
\left|\frac{\gamma_{k}(1-v)}{\gamma_{k}(v)}  \right|  \asymp_{k, a}  (1+ |\tau|)^{1-2a} \hspace{5pt}
\text{and} \hspace{5pt}
    \left|\frac{d^r}{dy^r} \left[G_{s}(y)y^{v-1}\right]\right| 
     \ll_{a, r, \sigma}  \left(\left(1 + |t|\right) \left(1 + |\tau|\right)\right)^r. 
\end{align}
Hence, if  $1< a<(k+1)/2$ and $r$ is an integer such that $2\le r< 2(a-1)$, then 
\begin{align}\label{largecest}
   \frac{X}{c} \,    \left| \int_{\R} \, G_{s}(y)  J_{k-1}\left(\frac{4\pi\sqrt{\ell yX}}{cD}\right)  e\left(-\frac{myX}{cD}\right) \, dy\right| 
    \ \ll  \     \left(1+ |t|\right)^r \frac{X^{a-r}}{|m|^r c^{2a-1-r}},
\end{align}
\begin{align}\label{smaest}
    \mathcal{I}_{s, >}(X;\ell)  \ll   \left(1+ |t|\right)^r X^{a-r}\sum_{c\gg \sqrt{X}}  \frac{1}{c^{2a-1-r}}   \ll \left(1+ |t|\right)^r X^{1-r/2},
\end{align}
where the implicit constants may depend on $\ell, D, k, \sigma, r, a$.

Next, consider $ \mathcal{I}_{s,\le}(X;\ell)$. After applying (\ref{asympJbess}),  we estimate the  oscillatory integral  
 \begin{align}
     \int_{\R} \, G_{s}\left(y\right)  W_{k}\left(\frac{2\sqrt{y\ell X}}{cD}\right) e\left(\frac{2\sqrt{y\ell X}-mXy}{cD}\right) \, dy. \nonumber
 \end{align}
To this end, we make use of  Lem. \ref{BKYibp} with the functions
 \begin{align*}
     w_{s}(y) \ := \ G_{s}\left(y\right)  W_{k}\left(\frac{2\sqrt{y\ell X}}{cD}\right)  \hspace{10pt} \text{ and } \hspace{10pt} h(y) \ := \ \frac{2\sqrt{y\ell X}-mXy}{cD}. 
 \end{align*}
Suppose  $y\in [1, 2]$ and $r\ge 2$. We observe the following bounds:
\begin{align*}
   |h'(y)| \gg \ \frac{|m|X}{cD},  \hspace{5pt} 
    |h^{(r)}(y)| \asymp_{r} \frac{\sqrt{\ell X}}{cD}  \ge  20, \hspace{5pt}  |w_{s}^{(r)}(y)| 
      \ll_{r,k, \sigma}  \left(1 + |t|\right)^r \left(\frac{2\sqrt{\ell X}}{cD}\right)^{-1/2},
\end{align*}
  which follows from the product rule and  (\ref{bessder}).  Now, apply Lem. \ref{BKYibp}  with
\begin{align*}
    W  =   \left(\frac{2\sqrt{\ell X}}{cD}\right)^{-1/2}, \hspace{10pt} V  =  \left(1+|t|\right)^{-1}, \hspace{10pt} H  =   \frac{\sqrt{\ell X}}{cD},  \hspace{10pt} G  =   1, \hspace{10pt} R  =  \frac{|m|X}{cD},
\end{align*}
we conclude, for any $c>0$, $m\neq 0$, $X>2\ell$, and $A>1$, that 
\begin{align*}
  \frac{X}{c} \,  \left| \int_{\R} \, w_{s}(y) e\left(h(y)\right) \, dy \right| 
    \ &\ll \ \frac{X^{3/4}}{\sqrt{c}} \,  \, \left(\frac{c\left(1+|t|\right)}{|m|X}\right)^{A},
\end{align*} 
\begin{align}\label{larest}
    \mathcal{I}_{s,\le}(X;\ell) \ \ll \ X^{3/4}  \, \left(\frac{\left(1+|t|\right)}{X}\right)^{A} \ \sum_{c\ll \sqrt{X}}  \ c^{A-1/2} \ \ll \ X \, \left(\frac{\left(1+|t|\right)}{\sqrt{X}}\right)^{A},
\end{align}
where the implicit constants depend on $k,\ell, D, A,\sigma$. The bound (\ref{larbd}) follows from taking $r=A=k-2$ and $a= (k+1)/2-\epsilon$  in (\ref{smaest})--(\ref{larest}). 
 \end{proof}

\begin{proof}[Proof of Prop. \ref{polygrAC}]
Let $\ell \ge 1$. For $\re s\gg  1$, we have
\begin{align}\label{dyaspl}
  \mathcal{A}_{\ell}(s) \ = \   \bigg(\sum_{\substack{2^u> 2\ell \\ u\ge -1} } \ + \ \sum_{\substack{2^u\le  2\ell \\ u\ge -1} }\bigg) \  \frac{\mathcal{I}_{s}(2^u;\ell)}{2^{us}}. 
\end{align}
Apply Lem.  \ref{weakt-asp} to the first sum of (\ref{dyaspl}), it follows that
\begin{align}\label{largell}
    \sum_{\substack{2^u> 2\ell \\ u\ge -1} } \ \left|\frac{\mathcal{I}_{s}(2^u;\ell)}{2^{us}}\right| \ \ll \  \left(1+|t|\right)^{k-2} \sum_{\substack{u\ge 0} } \frac{ (2^u)^{-\frac{k-4}{2}}}{2^{u\sigma}} \ \ll \ \left(1+|t|\right)^{k-2},
\end{align}
for any $\sigma > -(k-4)/2$ and $t\in \R$. On the other hand, observe that
\begin{align}\label{smallell}
\sum_{\substack{2^u\le  2\ell \\ u\ge -1} } \  \left|\frac{\mathcal{I}_{s}(2^u;\ell)}{2^{us}} \right| 
    \ \ll \ \sum_{n\le 4\ell} \ \  \sideset{}{^h}{\sum}_{f \in \mathcal{B}_{k}(D, \chi)} \ \left(|a_{f}(\ell)|^2 + |a_{f}(n)|^2\right) \ \ll \ 1.
\end{align}
The last estimate follows from Lem. \ref{Peterform}, the Weil bound for $S_{\chi}(m, n;c)$, (\ref{asympJbess}) and  (\ref{Jbessunibd}).

From (\ref{largell}) and (\ref{smallell}), both of the series on the right side of (\ref{dyaspl}) converge pointwise absolutely in the region $\re s >-(k-4)/2$, and can in fact be made uniform 
in every vertical strip $\sigma_{1} \le \sigma\le \sigma_{2}$ with $\sigma_{1}>-(k-4)/2$. The results follow. 
\end{proof}


\section{Proof of Theorem \ref{axiomFE}}\label{deduceFE}

\subsection{Mellin inversion}

Recall the set-up of Thm. \ref{axiomFE} and the definition of $\mathcal{B}_{k}^*(D, \chi)$ therein. Since $k\ge 4$, Prop. \ref{polygrAC} implies that $\mathcal{L}(s,f)$ and $\overline{\mathcal{L}}(s,f)$ admit analytic continuation up to $\re s>0$. For $g\in C_{c}^{\infty}(0, \infty)$, 
\begin{align}\label{Gprop}
    \text{ $\mathcal{G}(s):=\int_0^\infty g(x)x^{s-1}dx$ \  is entire and decays rapidly as $|\im s|\to\infty$. }
\end{align}

\begin{lemma}
With the same setting of Thm. \ref{axiomFE} and for any $\sigma\in(0,1)$, $\ell\geq1$, and $g\in C_{c}^{\infty}(0,\infty)$, we have 
\begin{align}\label{mainglobid}
    \int_{(\sigma)}  \mathcal{G}(s)   \bigg\{  \ \ 	\sideset{}{^h}{\sum}_{f \in \mathcal{B}^*_{k}(D, \chi)} \ \overline{a_{f}(\ell)} \, \bigg( \mathcal{L}(s, f)   -     i^{k} \chi(-1) \epsilon_{\chi} \overline{a_{f}(D)}\,  D^{\frac{1}{2}-s}  \frac{\gamma_{k}(1-s)}{\gamma_{k}(s)} \, \overline{\mathcal{L}}\left(1-s, f\right)\bigg) \bigg\}  \frac{ds}{2\pi i}  =   0. 
    \end{align}

\end{lemma}

\begin{proof}
Apply Mellin inversion, rearrange sums and integrals in (\ref{truncIsum}),  we have
\begin{align}
I_k(\ell; D, \chi)  
\ &= \   \, \int_{(3/2)} \, \mathcal{G}(s) \sideset{}{^h}{\sum}_{f \in \mathcal{B}^*_{k}(D, \chi)} \, \overline{a_{f}(\ell)}  \mathcal{L}(s,f) \ \frac{ds}{2\pi i}. \nonumber
\end{align} 
By Cor. \ref{polygrAC} and (\ref{Gprop}), we may shift the contour to $\re s =\sigma \in (0,1)$.  Recall Thm. \ref{FourVorthm}: 
\begin{align}
    I_k(\ell; D, \chi)  
	\ &= \ \frac{i^{k}\chi(-1) \overline{\chi}(\ell')\epsilon_{\chi}}{\sqrt{D}}  \ \  \sideset{}{^h}{\sum}_{f \in \mathcal{B}_{k}^*(D, \chi)} \, \sum_{n\ge 1}  \, (\mathcal{H}_{k}g)\left(\frac{n}{D}\right) \overline{a_{f}(\ell_{0}Dn)} a_{f}\left(\ell'\right). \nonumber
\end{align}
From (\ref{Hatrasdef}) and (\ref{bessgamm}), we have
\begin{align}\label{Hankmel}
	(\mathcal{H}_{k}g)\left(\frac{n}{D}\right)  \ = \   \frac{D}{n} \int_{(3/2)} \, \mathcal{G}(s) \frac{\gamma_{k}(1-s)}{\gamma_{k}(s)} \left(\frac{n}{D}\right)^{s}  \ \frac{ds}{2\pi i}. 
\end{align}
Using (\ref{Gprop}) and the estimate (\ref{ratiostir}), we shift the contour to $\re s = -1/2$ in (\ref{Hankmel}):
\begin{align}
	I_k(\ell; D, \chi)  
	 = i^k \, \chi(-1) \epsilon_{\chi} \ \ \sideset{}{^h}{\sum}_{f \in \mathcal{B}^*_{k}(D, \chi)}  \ a_{f}\left(\ell'\right)\overline{\chi}(\ell')   \int_{(-1/2)} \,   D^{\frac{1}{2}-s} \, \mathcal{G}(s) &\, \frac{\gamma_{k}(1-s)}{\gamma_{k}(s)}  \,   \sum_{n\ge 1} \, \frac{\overline{a_{f}}(\ell_{0}Dn)}{n^{1-s}} \,
  \frac{ds}{2\pi i}. \nonumber
\end{align}

We appeal to Assumption \ref{mainHecput} for $ \mathcal{B}_{k}^*(D,\chi)$. More precisely, 
\begin{align}
a_{f}\left(\ell'\right)\overline{\chi}\left(\ell'\right) \, = \, \overline{a_{f}}\left(\ell'\right), \hspace{5pt}
\overline{a_{f}}(\ell_{0}Dn) \, =  \, \overline{a_{f}}(\ell_{0})\overline{a_{f}}(D)\overline{a_{f}}(n), \hspace{5pt}
\overline{a_{f}}\left(\ell'\right)\overline{a_{f}}(\ell_{0}) \, = \, \overline{a_{f}}(\ell)  
\end{align}
for $n>0$ and $\ell= \ell_{0}\ell'$ with $\ell_{0}\mid D^{\infty}$ and $(\ell', D)=1$. 
It follows that
\begin{align}
	I_k(\ell; D, \chi)  
	\ = \ \sideset{}{^h}{\sum}_{f \in \mathcal{B}^*_{k}(D, \chi)}  \ \overline{a_{f}(\ell)}  \, \cdot \,  i^k \, \chi(-1) \epsilon_{\chi}  \overline{a_{f}(D)} \int_{(-1/2)} \,   D^{\frac{1}{2}-s} \, \mathcal{G}(s) &\, \frac{\gamma_{k}(1-s)}{\gamma_{k}(s)}  \,  \overline{\mathcal{L}}(1-s,f) \
  \frac{ds}{2\pi i}. \nonumber
\end{align}
Rearranging the sums and integrals above as
\begin{align}
    I_k(\ell; D, \chi)  
	 =     i^k \, \chi(-1) \epsilon_{\chi}  \overline{a_{f}(D)} \int_{(-1/2)} \,   D^{\frac{1}{2}-s} \, \mathcal{G}(s) &\, \frac{\gamma_{k}(1-s)}{\gamma_{k}(s)}  \, \sideset{}{^h}{\sum}_{f \in \mathcal{B}^*_{k}(D, \chi)}  \ \overline{a_{f}(\ell)} \ \overline{\mathcal{L}}(1-s,f) \
  \frac{ds}{2\pi i}. \nonumber
\end{align}
From (\ref{ratiostir}), (\ref{Gprop}), Cor. \ref{polygrAC}, the contour can be shifted to $\re s= \sigma\in (0,1)$. (\ref{mainglobid}) follows. 
\end{proof}


\subsection{Isolating a single form}\label{isosect}

Fix an orthogonal basis $\mathcal{B}^*_{k}(D, \chi):=\left\{ f_{1}, \ldots, f_{d}\right\}$ as in Assumption \ref{mainHecput}.  The vectors $ \underline{v_{i}} \ := \ \left( \overline{a_{f_{i}}}(1), \, \overline{a_{f_{i}}}(2), \, \overline{a_{f_{i}}}(3), \, \ldots \,  \right)$ for $i=1,\ldots, d$ are $\C$-linearly independent. Indeed, suppose $f:= \sum_{i=1}^{d} \alpha_{i} f_{i} \in S_{k}(D, \chi)$ and $\sum_{i=1}^{d} \overline{\alpha_{i}}\,  \underline{v_{i}}=0$ for some $\alpha_{i}\in \C$.  Then $\left(a_{f}(1), a_{f}(2), a_{f}(3), \, \ldots \,  \right)=0$. It follows from (\ref{conven}) that $f=0$. Hence, $\alpha_{1}=\cdots =\alpha_{d}=0$.  Then there exists $\ell_{1}< \cdots < \ell_{d}$ such that $A:= (\overline{a_{f_{i}}}(\ell_{j}))_{\substack{1\le i, j\le d}}$ is invertible.

By Cor. \ref{polygrAC},  the function $s  \mapsto  \sideset{}{^h}{\sum}_{f \in \mathcal{B}_{k}(D, \chi)} \ \overline{a_{f}(\ell_{j})} \mathcal{L}(s,f)$
admits a holomorphic continuation, say $G_{\ell_{j}}(s)$,  to the region  $\re s>-(k-4)/2$ for each $j=1,\ldots, d$. We have
\begin{align}\label{ACvecto}
    \left( \frac{\mathcal{L}(s,f_{1})}{||f_{1}||^2}, \ \ldots, \ \frac{\mathcal{L}(s,f_{d})}{||f_{d}||^2} \right)  \ =  \  \frac{(4\pi)^{k-1}}{\Gamma(k-1)}\left(G_{\ell_{1}}(s), \, \ldots, \, G_{\ell_{d}}(s)\right)A^{-1}.
\end{align}
From the right-hand side of (\ref{ACvecto}), we obtain an analytic continuation for $\mathcal{L}(s, f_{i})$  to $\re s>-(k-4)/2$ for each $i=1,\ldots, d$. The same argument works for $\overline{\mathcal{L}}(s, f_{i})$'s as well.

 Since the spectral identity (\ref{mainglobid}) holds for an arbitrary choice of $g\in C^{\infty}_{c}(0, \infty)$ and $\sigma\in (0,1)$, elementary analysis allows us to infer that
\begin{align*}
     \left( \frac{\mathcal{L}(s, f_{i})}{||f_{i}||^2}   -    i^{k} \chi(-1) \epsilon_{\chi} \overline{a_{f}(D)}\,  D^{\frac{1}{2}-s}  \frac{\gamma_{k}(1-s)}{\gamma_{k}(s)} \, \frac{\overline{\mathcal{L}}\left(1-s, f_{i}\right)}{||f_{i}||^2}\right)_{1\le i\le d} \, \cdot \,  A \ = \  (0,\ldots, 0) 
\end{align*}
for any $\ell \ge 1$ and on $0<\re s <1$. 
Since $A$ is invertible, (\ref{primnebFE}) for $f=f_{i}$ ($i= 1, \ldots, d$) follow.

\subsection{End Game}\label{fullcont}
In Sect. \ref{isosect}, we showed that $ \mathcal{L}(s, f_{i})$ admits a holomorphic continuation to the region $\re s> -(k-4)/2$. Also, we have
\begin{align}\label{moveFE}
  \mathcal{L}(s, f_{i})  \ =   \  i^{k} \chi(-1) \epsilon_{\chi} \overline{a_{f}(D)}\,  D^{\frac{1}{2}-s}  \frac{\gamma_{k}(1-s)}{\gamma_{k}(s)} \, \overline{\mathcal{L}}\left(1-s, f_{i}\right) 
\end{align}
upon restricting ourselves to the region $0< \re s< 1$. The right-hand side of (\ref{moveFE}) has been proven to be holomorphic on $\re s< k/2-1$ as well. As a result, we obtain an entire continuation for $\mathcal{L}(s, f_{i})$ (and similarly for $\overline{\mathcal{L}}(s,f_{i})$). Then (\ref{moveFE}) holds for all $s\in \C$ by analytic continuation. This completes the proof of Thm. \ref{axiomFE}.


\section{The case for trivial nebentypus (Theorem \ref{trivthm})}\label{trivsect}

\subsection{When \texorpdfstring{$D>1$}{D>1} is square-free}\label{sqfree} 
The calculations performed in Sect. \ref{mainbody} are valid through Sect. \ref{secPois}, but slight modifications are needed in Sect. \ref{arithsum}. From (\ref{secoPoi}), we have
\begin{align*}
  I_k(\ell; D)  &=   2\pi  \,  \sum_{n>0} \, (\mathcal{H}_{k}g)\left(\frac{n}{D}\right)\, 
 \sum_{\substack{c> 0\\ (c,D)=1 }}    \frac{S(n, \ell\overline{D}; c)}{cD}  J_{k-1}\left( \frac{4\pi \sqrt{\ell nD}}{cD}\right). 
\end{align*}
When $(c,D)=1$, we have
\begin{align}\label{trivtwmu}
    S(\ell_{0}Dn, \ell'; cD)  \ = \ \sideset{}{^*}{\sum}_{\beta\, (c)} \ e\left( \frac{\ell_{0}n\, \overline{\beta}+\ell'\, \overline{D}\beta}{c}\right) \  \sideset{}{^*}{\sum}_{\alpha \,(D)}\,  e\left(\frac{\ell' \, \overline{c}\alpha}{D}\right) \ = \   S(n, \ell\overline{D}; c) \  \sideset{}{^*}{\sum}_{\alpha \,(D)}\,  e\left(\frac{\alpha}{D}\right).  
\end{align}
 The last exponential sum is the \emph{Ramanujan sum} and can be evaluated as $\mu(D)$, see \cite[eq. (3.4)]{IK}. Since $D$ is square-free, we have $ S(n, \ell\overline{D}; c) \, = \,  \mu(D)S(\ell_{0}Dn, \ell'; cD)$
and
\begin{align*}
  I_k(\ell; D)    \ &= \  2\pi \mu(D) \,  \sum_{n>0} \, (\mathcal{H}_{k}g)\left(\frac{n}{D}\right)\, 
 \sum_{\substack{c> 0\\ (c,D)=1 }} \,   \frac{S(\ell_{0}Dn, \ell'; cD)}{cD}  J_{k-1}\left( \frac{4\pi \sqrt{\ell nD}}{cD}\right). 
\end{align*}
The condition $(c,D)=1$ in the $c$-sum can be dropped by Lem. \ref{basictwisKL}. Therefore, 
\begin{align*}
  I_k(\ell; D)    
 \ &= \  \mu(D)i^{k} \ \sideset{}{^h}{\sum}_{f \in \mathcal{B}_{k}(D)} \, a_{f}\left(\ell'\right)\, \sum_{n>0}  \,  \overline{a_{f}(\ell_{0}Dn)}(\mathcal{H}_{k}g)\left(\frac{n}{D}\right). 
\end{align*}
This proves Thm. \ref{trivthm}.(\ref{avgFourVortriv}). Following the argument of Sect. \ref{deduceFE}, we deduce Thm. \ref{trivthm}.(\ref{triFEpt}).

\subsection{When \texorpdfstring{$D=1$}{D=1}}
(\ref{resulPetPo}) must be modified. Indeed, when $m=0$,  we have  $x \equiv 0 \  (c)$ and $(x,c)=1$ in (\ref{evagammsum}),  giving $c=1$. The dual zeroth frequency is given by $i^{-k}(\mathcal{H}_{k}g)(\ell)$. Thus, 
\begin{align*}
		I_k(\ell)  = g(\ell)  +  i^{-k}(\mathcal{H}_{k}g)(\ell) 
   +  2\pi i^{-k}  \sum_{c=1}^{\infty} \frac{1}{c}  \sum_{\substack{m\neq 0 \\ (m, c)=1}}      e\left(-\frac{\ell \overline{m}}{c}\right)   \int_{\R} g(y)  J_{k-1}\left(\frac{4\pi\sqrt{\ell y}}{c}\right)  e\left(-\frac{my}{c}\right)  dy.
\end{align*}
Following the calculations done in Sect. \ref{sect:step2}--\ref{poisprep}, eq. (\ref{compforPois}) becomes
\begin{align*}
    I_k(\ell)  =  i^{-k}(\mathcal{H}_{k}g)(\ell)  +  
    2\pi   \sum_{c\in \mathbb{Z}} \, \sum_{\substack{m> 0 \\ (m, c)=1}}    e\left(\frac{\ell \, \overline{c}}{m}\right)   \int_{0}^{\infty} (\mathcal{H}_{k}g)\left( mx\right)    J_{k-1}\left( 4\pi\sqrt{\frac{\ell x}{m}}\right) e\left(cx\right) dx.
\end{align*}
 Applying Poisson summation as in Sect. \ref{secPois}, we obtain
\begin{align*}
  I_k(\ell)  &=   i^{-k}(\mathcal{H}_{k}g)(\ell)  +  i^k \sum_{n>0}  (\mathcal{H}_{k}g)\left(n\right) \cdot  2\pi i^{-k} 
 \sum_{\substack{c> 0 }}    \frac{S(n, \ell; c)}{c}      J_{k-1}\left( \frac{4\pi \sqrt{\ell n}}{c}\right). 
\end{align*}
The discussion of Sect. \ref{arithsum} is not necessary in the present case, except notice that $S(n, \ell; c)= S( \ell,n; c)$. As in Sect. \ref{Petinrev}, we apply the Petersson formula in reverse:
\begin{align*}
  I_k(\ell)     &=   i^{-k}(\mathcal{H}_{k}g)(\ell)  +  i^k \sum_{n>0} \, (\mathcal{H}_{k}g)\left(n\right) \bigg\{  \ \  \sideset{}{^h}{\sum}_{f \in \mathcal{B}_{k}(1)} \,  \overline{a_{f}(\ell)} a_{f}\left(n\right) -   \delta(n=\ell)\bigg\}. 
\end{align*}
Since $2\mid k$,   the dual zeroth frequency from the first application of Poisson summation cancels out the diagonal term from the second application of the Petersson formula. So,
\begin{align*}
  I_k(\ell)   =    i^k \sideset{}{^h}{\sum}_{f \in \mathcal{B}_{k}(1)} \,  \overline{a_{f}(\ell)} \ \sum_{n>0} \, a_{f}\left(n\right)(\mathcal{H}_{k}g)\left(n\right). 
\end{align*}
Then Thm. \ref{trivthm}.(\ref{ptD1ACFE}) follows readily from the discussions of Sect. \ref{deduceFE}. Thm. \ref{trivthm} follows.



\printbibliography

\end{document}